\documentclass[twoside]{amsart}
\usepackage{amsfonts,amsthm}
\usepackage{enumerate}
\usepackage{amsmath}
\usepackage{amssymb}
\usepackage{amsxtra}
\usepackage{latexsym}

\theoremstyle{plain}
\newtheorem*{thm A}{Theorem~A}
\newtheorem*{thm B}{Theorem~B}
\newtheorem*{thm C}{Theorem~C}
\newtheorem*{main 1}{Theorem~1}
\newtheorem*{main 2}{Theorem~2}
\newtheorem*{main 3}{Theorem~3}
\newtheorem*{coro 1}{Corollary~1}
\newtheorem*{coro 2}{Corollary~2}
\newtheorem*{pro A}{Proposition~A}
\newtheorem*{pro B}{Proposition~B}
\newtheorem*{rem}{Remark}

\newtheorem{theorem}{Theorem}[section]

\newtheorem{lemma}[theorem]{Lemma}

\def \GBt{G_2({\mathbb C}^{m+2})}
\def \GBo{G_2({\mathbb C}^{m+1})}

\def \QP{{\mathcal Q}^{\bot}}
\def \Q{\mathcal Q}
\def \N{\nabla}
\def \EN{{\eta}_{\nu}}
\def \EoK{{\eta}_1({\xi})}

\def \Kt{{\xi}_2}
\def \Ks{{\xi}_3}
\def \Ph{\phi}
\def \e{\eta}
\def \E{\eta}
\def \al{\alpha}
\def \xo{{\xi}_1}
\def \xtw{{\xi}_2}
\def \po{\phi_{1}}
\def \pn{\phi_{\nu}}
\def \etw{\eta_{2}}
\def \eth{\eta_{3}}
\def \xtw{{\xi}_2}
\def \xth{{\xi}_3}
\def \x{\xi}
\def \p{\phi}
\def \xt{\xi_{2}}
\def \xh{\xi_{3}}
\def \GBt{G_2({\mathbb C}^{m+2})}
\def \GBo{G_2({\mathbb C}^{m+1})}
\def \gtw{\widehat \nabla ^{(k)}}

\def \xm{\xi_{\mu}}
\def \xmo{\xi_{\mu+1}}
\def \xmt{\xi_{\mu+2}}

\def \qmo{q_{\mu+1}}
\def \qmt{q_{\mu+2}}

\def \pmx{\phi_{\mu}\xi}

\def \si{\sigma}
\def \vr{\tau}

\def \Xq{X_{\mathfrak h}}

\begin{document}

\title[g-Tanaka-Webster Reeb parallel Ricci tensor] {Real hypersurfaces in complex two-plane Grassmannians with Reeb parallel Ricci tensor in generalized Tanaka-Webster connection}

\vspace{0.2in}

\author[H. Lee, Y.J. Suh and C. Woo]{Hyunjin Lee, Young Jin Suh and Changhwa Woo}

\address{\newline
Hyunjin Lee
\newline The Center for Geometry and its Applications,
\newline Pohang University of Science \& Technology,
\newline Pohang 790-784, REPUBLIC OF KOREA}
\email{lhjibis@hanmail.net}

\address{\newline
Young Jin Suh and Changhwa Woo
\newline Department of Mathematics,
\newline Kyungpook National University,
\newline Daegu 702-701, REPUBLIC OF KOREA}
\email{yjsuh@knu.ac.kr} \email{legalgwch@knu.ac.kr}


\thanks{* This work was supported by Grant Proj. No. NRF-2011-220-C00002 from National Research Foundation of Korea.
The first author by Grant Proj. No. NRF-2012-R1A1A3002031, the second by Grant Proj. No. NRF-2012-R1A2A2A01043023. And the third author supported by NRF Grant funded by the Korean Government (NRF-2013-Fostering Core Leaders of Future Basic Science Program).}

\begin{abstract}
There are several kinds of classification problems for real hypersurfaces in complex two-plane Grassmannians $\GBt$. Among them, Suh classified Hopf hypersurfaces $M$ in $\GBt$ with Reeb parallel Ricci tensor in Levi-Civita connection. In this paper, we introduce a new notion of generalized Tanaka-Webster Reeb parallel Ricci tensor for $M$ in $\GBt$. By using such parallel conditions, we give complete classifications of Hopf hypersurfaces in $\GBt$.
\end{abstract}
\maketitle

\section*{Introduction}
\setcounter{equation}{0}
\renewcommand{\theequation}{0.\arabic{equation}}
\vspace{0.13in}

In this paper, let $M$ represent a real hypersurface in $\GBt$, $m \geq 3$, and~$S$ denote the Ricci tensor of $M$. Hereafter unless otherwise stated, we consider that $X,Y$, and $Z$ are any tangent vector fields on $M$.  Let $W$ be any tangent vector field on the distribution $\mathfrak h=\{X \in TM |\, X \bot \xi\}$. $k$ stands for a non-zero constant real number.

The classification of real hypersurfaces in Hermitian symmetric space is one of interesting parts in the field of differential geometry. Among them, we introduce a complex two-plane Grassmannian~$\GBt$ defined by the set of all complex two-dimensional linear subspaces in ${\mathbb C}^{m+2}.$ It is a kind of Hermitian symmetric space of compact irreducible type with rank~$2$. Remarkably, the manifolds are equipped with both a K\"{a}hler structure $J$ and a quaternionic K\"{a}hler structure ${\mathfrak J}$ satisfying $JJ_{\nu}=J_{\nu}J$ $(\nu=1,2,3)$ where $\{J_{\nu}\}_{\nu=1,2,3}$ is an orthonormal basis of $\mathfrak J$. When $m=1$, $G_2({\Bbb C}^3)$ is isometric to the two-dimensional complex projective space ${\mathbb C}P^2$ with constant holomorphic sectional curvature eight. When $m=2$, we note that the isomorphism $\text{Spin}(6) \simeq \text{SU}(4)$ yields an isometry between $G_2({\Bbb C}^4)$ and the real Grassmann Manifold $G_2^+({\mathbb R}^6)$ of oriented two-dimensional linear
subspaces in ${\Bbb R}^6$. In this paper we assume $m$ is not less than $3$. (see \cite{BS1}).

\vskip 3pt

Let $N$ be a local unit normal vector field of $M$. Since $\GBt$ has the K\"{a}hler structure $J$, we may define a {\it Reeb vector field} $\xi =-JN$ and a $1$-dimensional distribution $[\xi]=\text{Span}\{\,\xi\}$. The Reeb vector field~$\xi$ is said to be a {\it Hopf} if it is invariant under the shape operator $A$ of $M$. The $1$-dimensional foliation of $M$ by the integral curves of $\xi$ is said to be a {\it Hopf foliation} of $M$. We say that $M$ is a {\it Hopf hypersurface} if and if the Hopf foliation of $M$ is totally geodesic. By the formulas in \cite[Section~$2$]{LS}, it can be easily seen that $\xi$ is Hopf if and only if $M$ is Hopf.

\vskip 3pt

From the quaternionic K\"{a}hler structure $\mathfrak J$ of $\GBt$, there naturally exists {\it almost contact 3-structure} vector field $\xi_{\nu}=-J_{\nu}N$, $\nu=1,2,3$. Put $\QP = \text{Span}\{\,\xi_1, \xi_2, \xi_3\}$, which is a 3-dimensional distribution in a tangent vector space $T_{x}M$ of $M$ at $x \in M$. In addition, $\Q$ stands for the orthogonal complement of $\QP$ in $T_{x}M$. It becomes the quaternionic maximal subbundle of $T_{x}M$. Thus the tangent space of $M$ consists of the direct sum of $\Q$ and $\QP$ as follows: $T_{x}M =\Q\oplus \QP$.

\vskip 3pt

For two distributions $[\xi]$ and $\QP$ defined above, we may consider two natural invariant geometric properties under the shape operator $A$ of $M$, that is, $A [\xi] \subset [\xi]$ and $A\QP \subset \QP$. By using the result of Alekseevskii \cite{Al-01}, Berndt and Suh~\cite{BS1} have classified all real hypersurfaces with two natural invariant properties in $\GBt$ as follows:

\begin{thm A}
Let $M$ be a real hypersurface in $\GBt$, $m \geq 3$. Then both $[\xi]$ and $\QP$ are invariant under the shape operator of $M$ if and only if
\begin{enumerate}[\rm(A)]
\item {$M$ is an open part of a tube around a totally geodesic $\GBo$ in $\GBt$, or} \
\item {$m$ is even, say $m = 2n$, and $M$ is an open part of a tube around a totally geodesic ${\mathbb H}P^n$ in $\GBt$}.
\end{enumerate}
\end{thm A}

\noindent In the case~$(A)$, we say $M$ is of Type~$(A)$. Similarly in the case~$(B)$ we say $M$ is of Type~$(B)$. Using Theorem~$\rm A$, geometricians have given characterizations for Hopf hypersurfaces in $\GBt$ with geometric quantities, shape operator, normal (or structure) Jacobi operator, Ricci tensor, and so on. Actually, Lee and Suh~\cite{LS} gave a characterization for a real hypersurface of Type~$(B)$ as follows:
\begin{thm B}
Let $M$ be a Hopf hypersurface in $\GBt$, $m \geq 3$. Then $\xi$ belongs to the
distribution $\Q$ if and only if $M$ is locally congruent to an open part of a tube around a totally geodesic ${\mathbb H}P^n$ in $\GBt$, $m=2n$. In other words, $M$ is locally congruent to a real hypersurface of Type~(B).
\end{thm B}

\noindent In particular, there are various well-known results with respect to $S$ on Hopf hypersurfaces in $\GBt$. From such a point of view, Suh~\cite{S02} gave a characterization of a model space of Type~$(A)$ in $\GBt$ under the condition~ $S \p = \p S$ where $\p$ denotes the structure tensor field of $M$. In~\cite{S} and \cite{S01}, he also considered the parallelism of Ricci tensor with respect to the Levi-Civita connection and gave, respectively,
\begin{thm C}\cite{S01}
Let $M$ be a real hypersurface in $\GBt$, $m \geq 3$ with non-vanishing geodesic Reeb flow. If the Ricci tensor is Reeb parallel, $\N_{\x}S=0$. Then $M$ is locally congruent to one of the following:
\begin{enumerate}[\rm(i)]
\item {a tube over a totally geodesic $G_2({\mathbb C}^{m+1})$ in $G_2({\mathbb C}^{m+2})$ with radius $r \neq \frac{\pi}{4 \sqrt{2}}$, or} \
\item {a tube over a totally geodesic ${\mathbb H}P^n$, $m = 2n$, in $G_2({\mathbb C}^{m+2})$ with radius $r$ such that $\cot^{2}(2r) = \frac{1}{2m-1}$ and $\x$-parallel eigenspaces $T_{\cot r}$ and $T_{\tan r}$ }.
\end{enumerate}
\end{thm C}

Motivated by these works, we define the notion of Reeb parallel Ricci tensor with respect to the generalized Tanaka-Webster connection for a real hypersurface $M$ in $\GBt$. In order to do this, we first define the generalized Tanaka-Webster connection~$\gtw$ on $M$ given by
\begin{equation*}
\gtw_{X}Y = \nabla_X Y+g(\phi AX,Y)\xi -\eta(Y)\phi AX-k\eta(X)\phi Y,
\end{equation*}
where $k$ is a non-zero real number (see~\cite{CH1},~\cite{CH2},~ \cite{JKLS}). Hereafter, unless otherwise stated, a GTW connection means a generalized Tanaka-Webster connection. In addition, we put
\begin{equation*}
F_{X}^{(k)}Y = g(\phi AX,Y)\xi -\eta(Y)\phi AX-k\eta(X)\phi Y.
\end{equation*}
Then the operator $F_{X}^{(k)}$ becomes a skew-symmetric (1,1) type tensor, that is, $g(F_{X}^{(k)}Y, Z)=-g(Y, F_{X}^{(k)}Z)$ for any tangent vector fields $X, Y$, and~$Z$ on $M$ and said to be {\it Tanaka-Webster} (or {\it $k$-th-Cho}) {\it operator} with respect to $X$.

\vskip 3pt

Related to this connection, the Ricci tensor $S$ is said to be {\it generalized Tanaka-Webster Reeb parallel} (in short, {\it GTW-Reeb parallel}) if the covariant derivative in GTW connection $\gtw$ of $S$ along $\x$ is vanishing, that is, $(\gtw_{\x}S)Y=0$. From this, we naturally see that this notion is weaker than generalized Tanaka-Webster parallel (shortly, GTW-parallel) Ricci tensor, that is, $(\gtw_{X} S)Y =0$. Recently, P\'{e}rez and Suh~\cite{PS2} proved the non-existence of Hopf hypersurfaces in $\GBt$, $m \geq 3$, with GTW-parallel Ricci tensor. From such a viewpoint, we assert:
\begin{main 1}
Let $M$ be a Hopf hypersurface in complex two-plane Grassmannians $\GBt$, $m \geq 3$, with $\al=g(A\xi,\xi)\neq 2k$. The Ricci tensor $S$ of $M$ is GTW-Reeb parallel if and only if $M$ is locally congruent to one of the following:
\begin{enumerate}[\rm(i)]
\item {a tube over a totally geodesic $G_2({\mathbb C}^{m+1})$ in $G_2({\mathbb C}^{m+2})$ with radius $r$ such that  $r \neq \frac{1}{2\sqrt{2}}\cot^{-1}(\frac{k}{\sqrt{2}})$, or}
\item {a tube over a totally geodesic ${\mathbb H}P^n$, $m = 2n$, in $G_2({\mathbb C}^{m+2})$ with radius $r$ such that $r=\frac{1}{2}\cot^{-1}(\frac{-k}{4(2n-1)})$.}
\end{enumerate}
\end{main 1}

For the case $\al=2k$, the Reeb vector field $\x$ of Hopf hypersurface $M$ with GTW-Reeb parallel Ricci tensor belongs to either $\Q$ or $\QP$. So, for the case $\x \in \QP$, we obtain that the trace $h$ of the shape operator $A$ is constant along $\x$, that is, $\x h=0$. In addition for the case $\x \in \Q$ we have the following:
\begin{coro 1}
Let $M$ be a real hypersurface in complex two-plane Grassmannians $\GBt$, $m \geq 3$, with GTW-Reeb parallel Ricci tensor for $\al=2k$. If $\x$ belongs to the distribution $\Q$, then $M$ is locally congruent to an open part of a tube around a totally geodesic $\mathbb HP^{n}$, $m=2n$, in $\GBt$ with radius $r$ such that $r=\frac{1}{2}\tan^{-1}\sqrt{2n-1}$.
\end{coro 1}

\noindent On the other hand, we consider the notion of GTW-Reeb parallel Ricci tensor on $\mathfrak h$, that is, $(\gtw_{\x}S)W=0$ for any $W \in \mathfrak h$. Then by virtue of Theorem~$\rm C$ for the case $\al = 2k$, we assert the following:
\begin{main 2}\label{main 2}
Let $M$ be a Hopf hypersurface in complex two-plane Grassmannians $\GBt$, $m \geq 3$, with $\al=2k$. The Ricci tensor of $M$ satisfies the Reeb parallelism on $\mathfrak h$ in both GTW and Levi-Civita connections, that is, $(\gtw_{\x}S)W=0$ and $(\N_{\x}S)W=0$ for any $W \in \mathfrak h$ if and only if $M$ is locally congruent to an open part of a tube around a totally geodesic $\GBo$ in $\GBt$ with radius $r$ such that $ r = \frac{1}{2\sqrt{2}} \cot^{-1}(\frac{k}{\sqrt{2}})$.
\end{main 2}

Moreover, as a generalization of the assumption $\gtw_{\x}S =0= \N_{\x}S$ on $\mathfrak h$ in Theorem~$2$, we want to consider that $\gtw_{\x}S = \N_{\x}S$, that is, the Reeb parallel Ricci tensor in GTW connection coincides with the Reeb parallel Ricci tensor in Levi-Civita connection. This condition has a geometric meaning such that $S$ commutes with the Tanaka-Webster operator $F_{\x}$, that is, $S \cdot F_{\x} = F_{\x} \cdot S$. This meaning gives any eigenspaces of $S$ are invariant by the Tanaka-Webster operator $F_{\x}$. From such a point of a view, we have the following:
\begin{main 3}
Let $M$ be a Hopf hypersurface in complex two-plane Grassmannians $\GBt$, $m \geq 3$. Then $\gtw_{\x}S = \N_{\x}S$ if and only if $M$ is locally congruent to an open part of a tube around a totally geodesic $\GBo$ in $\GBt$.
\end{main 3}

\noindent But for the case where the derivative of the Ricci tensor in GTW connection is equal to the derivative in Levi-Civita connection, that is, $\gtw_{X} S = \N_{X} S$ for any $X \in TM$, we assert the following:
\begin{coro 2}
There does not exist any Hopf hypersurface in complex two-plane Grassmannians $\GBt$, $m \geq 3$, satisfying $(\gtw_{X}S)Y = (\N_{X}S)Y$ for arbitrary tangent vector fields $X$ and $Y$ on $M$.
\end{coro 2}

\noindent Obviously, we know that the condition $\gtw_{X} S = \N_{X} S$ has a geometric meaning that any eigenspaces of $S$ are invariant by the Tanaka-Webster operator $F_{X}$. Recently, P\'{e}rez and Suh~\cite{PeSuh} investigated the Levi-Civita and GTW covariant derivatives for the shape operator or the structure Jacobi operator of real hypersurfaces in complex projective space $\mathbb{C} P^{m}$. Moreover, in \cite{JLS} Jeong, Lee and Suh gave a characterization of Hopf hypersurfaces in $\GBt$ with $\gtw A = \N A$.

\vskip 3pt

In this paper, we refer \cite{Al-01}, \cite{BS1}, \cite{JMPS}, \cite{LS}, \cite{PeSuWa} and \cite{S02} for Riemannian geometric structures of $\GBt$ and its geometric quantities, respectively.
In order to get our results, in sections~\ref{section 3} we will give the fundamental formulas related to the Reeb parallel Ricci tensor. In section~\ref{section 4}, we want to give a complete proof of Theorem~$1$ for $\al=g(A\xi, \xi) \neq 2k$. In section~\ref{section 5} we will consider the case $\al = 2k$ and give a proof of Corollary~$1$ and Theorem~$2$. Finally, in section~\ref{section 6} we will give a complete proof of Theorem~$\rm 3$ and Corollary~$2$.

\vspace{0.15in}

\section{GTW-Reeb parallel Ricci tensor}\label{section 3}
\setcounter{equation}{0}
\renewcommand{\theequation}{3.\arabic{equation}}
\vspace{0.13in}

From \cite{PS}, the Ricci tensor $S$ of a real hypersurface $M$ in $\GBt$, $m \geq 3$, is given by
\begin{equation} \label{3.1}
\begin{split}
SX & = \sum_{i=1}^{4m-1} R(X,e_i)e_i \\
   & = (4m+7)X - 3\eta(X)\xi + hAX - A^2 X \\
     & \quad + \sum_{\nu =1}^3 \{ - 3 \eta_{\nu}(X)\xi_{\nu} + \eta_{\nu}(\xi)\phi_{\nu}\phi X-\eta(\phi_{\nu}X)\phi_{\nu}\xi -\eta(X)\eta_{\nu}(\xi)\xi_{\nu} \}
\end{split}
\end{equation} where $h$ denotes the trace of the shape operator $A$, that is, $h=\text{Tr}A$.

And we also have
\begin{equation}\label{3.3}
\begin{split}
(\nabla_{X}S)Y & = - 3g(\phi AX, Y) \xi - 3\eta(Y) \phi AX \\
                 & \quad - 3\sum_{\nu=1}^{3} \Big\{ g(\phi_{\nu}AX, Y) \xi_{\nu} + \eta_{\nu} (Y) \phi_{\nu}AX \Big\}\\
                 &\quad + \sum_{\nu=1}^{3}\Big\{2g(\phi AX, \xi_{\nu})\phi_{\nu}\phi Y + g(AX, \phi_{\nu}\phi Y)\phi_{\nu}\xi \\
                 &\quad \ \ - \eta(Y) g(AX, \xi_{\nu})\phi_{\nu}\xi + \eta_{\nu}(\phi Y) g(AX, \xi) \xi_{\nu}-\eta_{\nu}(\phi Y) \phi_{\nu} \phi AX \\
                 &\quad \ \ - \eta(Y) g(\phi AX, \xi_{\nu})\xi_{\nu}-\eta(Y) g(\phi_{\nu}AX, \xi)\xi_{\nu}\Big\}\\
                 &\quad +(Xh)AY +h(\nabla_{X}A)Y - (\nabla_{X}A)AY -A(\nabla_{X}A)Y.
\end{split}
\end{equation}

\noindent Substituting $X=\x$ into \eqref{3.3} and using the condition that $M$ is Hopf, that is, $A \x = \al \x$, we get
\begin{equation}\label{3.4}
\begin{split}
(\nabla_{\x}S)Y & = -4 \al \sum_{\nu=1}^{3} \Big\{ g(\phi_{\nu}\x, Y) \xi_{\nu} + \eta_{\nu} (Y) \phi_{\nu}\x \Big\}
                    +(\x h)AY\\
                  & \quad  +h(\nabla_{\x}A)Y - (\nabla_{\x}A)AY -A(\nabla_{\x}A)Y.
\end{split}
\end{equation}

In this section we assume that $M$ is a Hopf hypersurface in $\GBt$ with GTW-Reeb parallel Ricci tensor, that is, $S$ satisfies:
\begin{equation}\label{C-1}
(\hat{\nabla}_\xi^{(k)}S)X =0.
\tag{C-1}
\end{equation}
By the definition of GTW connection $\gtw$, the covariant derivative of $S$ with respect to the GTW connection along $\x$ becomes
\begin{equation} \label{3.5}
\begin{split}
(\widehat{\nabla}_\xi^{(k)}S)X & = \widehat{\nabla}_\xi^{(k)}(SX)-S(\widehat{\nabla}_\xi^{(k)}X)    \\
                               & = \nabla_\xi (SX)+g(\phi A\xi,SX)\xi -\eta(SX)\phi A\xi-k\eta(\xi)\phi SX   \\
                                 & \quad \quad -S(\nabla_\xi X)-g(\phi A\xi,X)S\xi +\eta(X)S\phi A\xi+k\eta(\xi)S\phi X \\
                               & =  (\nabla_\xi S)X -k \phi SX + k S\phi X.
\end{split}
\end{equation}
Thus the condition~\eqref{C-1} is equivalent to
\begin{equation} \label{3.6}
(\nabla_\xi S)X = k\phi SX -kS\phi X,
\end{equation}
it yields
\begin{equation} \label{3.7}
\begin{split}
& 4(k-\al) \sum_{\nu =1}^3 \Big \{\eta_{\nu}(\p X)\x_{\nu} -  \EN(X)\pn\x \Big\} \\
& \ \  =  (\x h)AX +h(\nabla_{\x}A)X - (\nabla_{\x}A)AX -A(\nabla_{\x}A)X - kh\Ph AX\\
& \quad \quad \ \  + k\Ph A^2 X + khA\Ph X - kA^2 \Ph X\\
\end{split}
\end{equation}
from \eqref{3.1}, \eqref{3.3} and \cite[Section~$2$]{LCW}.

\vskip 5pt
Using these equations, we prove that $\x$ belongs to either $\Q$ or $\QP$, where $M$ is a Hopf hypersurface in $\GBt$ with GTW-Reeb parallel Ricci tensor.

\begin{lemma}\label{lemma 3.1}
Let $M$ be a Hopf hypersurface in $\GBt$, $m \geq 3$. If $M$ has GTW-Reeb parallel Ricci tensor, then $\xi$ belongs to either $\Q$ or $\QP$.
\end{lemma}

\begin{proof}
In order to prove this lemma, we put
\begin{equation}\label{**}
\xi = \eta(X_{0})X_{0}+\eta(\xi_{1})\xi_{1}
\tag{**}
\end{equation}
for some unit vectors $X_{0} \in \Q$ and $\xi_{1} \in \QP$. Putting $X=\xi$ in (\ref{3.7}), by \eqref{**} and basic formulas in \cite[Section 2]{LCW}, it follows that
\begin{equation} \label{3.8}
4(k - \al) \EoK\po\x = \al(\x h)\x - h(\x \al)\x - 2\al (\x \al)\x,
\end{equation}
where we have used $(\N_{\x}A)\x = (\x \al) \x$ and $(\N_{\x}A)A\x=\al (\x \al) \x$.

\noindent Taking the inner product of \eqref{3.8} with $\po\x$, we have
\begin{equation}\label{3.10}
4(k -\al)\EoK\eta^{2}(X_0)=0,
\end{equation}
because of $\e^{2}(X_{0}) + \e^{2}(\xo) =1$. From this, we have the following three cases.

\vskip 3pt
\noindent {\bf Case 1\,:} \quad $\al = k$.
\vskip 3pt
For this case, we see that $\al$ becomes a non-zero real number. Using the equation in \cite[Lemma~$1$]{BS1}, we assert that $\xi$ belongs to either $\Q$ or $\QP$.

\vskip 3pt
\noindent {\bf Case 2\,:} \quad $\eta(\xi_{1})=0$.
\vskip 3pt

By the notation \eqref{**}, we see that $\xi$ belongs to $\Q$.

\vskip 3pt
\noindent {\bf Case 3\,:} \quad $\E(X_0)=0$.
\vskip 3pt

This case implies that $\xi$ belongs to $\QP$ from \eqref{**}.

\vskip 5pt

\noindent Accordingly, summing up these cases, the proof of our Lemma is completed.
\end{proof}

\vspace {0.15in}


\section{Proof of Theorem~$1$}\label{section 4}
\setcounter{equation}{0}
\renewcommand{\theequation}{4.\arabic{equation}}
\vspace{0.13in}

In this section, let $M$ be a Hopf hypersurface, $\alpha\neq 2k$, in
$\GBt$ with GTW-Reeb parallel Ricci
tensor. Then by Lemma~\ref{lemma 3.1} we shall divide our
consideration in two cases depending on $\xi$
belongs to either $\QP$ or $\Q$,
respectively.

\vskip 3pt

First of all, if we assume $\x \in \Q$, then a Hopf hypersurface in $\GBt$, $m \geq 3$, with GTW-Reeb parallel Ricci tensor and $\al =g(A\x, \x) \neq 2k$ is locally congruent to a real hypersurface of Type~$(B)$ by virtue of Theorem~$\rm B$ given in the introduction.

\vskip 3pt

Next let us consider the case, $\xi \in \QP$. Accordingly, we may put $\xi=\xi_1$. Since $M$ is a Hopf hypersurface with GTW-Reeb parallel Ricci tensor, the equation~\eqref{3.7} becomes
\begin{equation} \label{4.1}
\begin{split}
&(\xi h)AX+h(\N_\x A)X-(\N_\x A)AX-A(\N_\x A)X \\
&=k(h\p AX-\p A^{2}X-h A\p X+A^{2}\p X).
\end{split}
\end{equation}
From the Codazzi equation~\cite[Section~$2$]{LCW} and differentiating $A\xi=\alpha\xi$, we obtain
\begin{equation*} \label{4.2}
\begin{split}
(\N_\x A)X & = (\N_X A)\xi+\p X+\po X+2\eth(X)\xt-2\etw(X)\xth \\
           & = (X\alpha)\x+\al \p AX-A\p AX +\p X+\po X+2\eth(X)\xt-2\etw(X)\xth.
\end{split}
\end{equation*}

\noindent Using the equation~\cite[Lemma 2.1]{LCW} and the previous one, we get
\begin{equation*} \label{4.4}
(\N_\x A)X=\frac{\al}{2}\p AX-\frac{\al}{2}A\p X+ (\x\al)\eta (X)\x.
\end{equation*}
Therefore from this, \eqref{4.1} can be written as
\begin{equation} \label{4.6}
(\xi h)AX + \tilde{\kappa} h\p AX-\tilde{\kappa} h A\p X +(h-2\al)(\x\al)\e(X)\x-\tilde{\kappa} \p A^{2}X+ \tilde{\kappa} A^{2}\p X=0,
\end{equation}
where $\tilde{\kappa} = (\frac{\al}{2}-k)$.

Since $\tilde{\kappa} \neq 0$ is equivalent to the given condition $\al \neq 2k$, \eqref{4.6} yields
\begin{equation} \label{4.7}
\frac{(\xi h)}{\tilde{\kappa}}AX + h\p AX - h A\p X + \frac{(h-2\al)}{\tilde{\kappa}}(\x\al) \e(X)\x-\p A^{2}X+A^{2}\p X=0.
\end{equation}

Now we consider the case $\x h=0$. Then \eqref{4.7} can be reduced to
\begin{equation}\label{4.8}
h\p AX - h A\p X + \frac{(h-2\al)}{\tilde{\kappa}}(\x\al)\e(X)\x-\p A^{2}X+A^{2}\p X=0.
\end{equation}
Taking the inner product of \eqref{4.8} with $\x$, we have $\frac{(h-2\al)}{\tilde{\kappa}}(\x\al)\e(X)=0$. Thus (\ref{4.8}) becomes
\begin{equation} \label{4.9}
\begin{split}
h\p AX-\p A^{2}X-h A\p X+A^{2}\p X=0.
\end{split}
\end{equation}

On the other hand, from the equation~\eqref{3.1} we calculate
\begin{equation*}
S\p X - \p SX = h A\p X - A^{2} \p X - h \p AX + \p A^{2}X,
\end{equation*}
then by \eqref{4.9} it follows that $S\p X = \p SX$ for any tangent vector field $X$ on $M$. Hence, by Suh~\cite{S02} we assert that $M$ satisfying our assumptions must be a model space of Type~$(A)$.

\vskip 3pt

We now assume $\xi h\neq 0$. Putting $\si=\frac{(\xi h)}{\tilde{\kappa}}(\neq 0)$ and
$\vr=\frac{(h-2\al)}{\tilde{\kappa}} (\x \al)$, the equation~(\ref{4.7}) becomes
\begin{equation} \label{4.10}
\si AX + h \p AX - h A\p X + \vr \e(X)\x - \p A^{2}X + A^{2}\p X=0.
\end{equation}

\noindent Applying $\p$ to \eqref{4.10} and replacing $X$ by $\p X$ in \eqref{4.10}, respectively, we
get the following two equations:
\begin{equation*} \label{4.11}
\si \p AX - h AX + h \al \e(X) \x - h \p A\p X + A^{2}X - \al^{2}\e(X)\x + \p A^{2}\p X=0
\end{equation*}
and
\begin{equation*} \label{4.12}
\si A\p X + h\p A\p X + h AX - h \al\e(X)\x - \p A^{2}\p X - A^{2}X + \al^{2}\e(X)\x=0.
\end{equation*}

\noindent Summing up the above two equations, we obtain $\p A+ A\p=0$. Thus from this, the equation~\eqref{4.10} implies
\begin{equation*} \label{4.14}
\si AX+ 2h\p AX + \vr \e(X)\x=0.
\end{equation*}


Let us $\Xq$ be the orthogonal projection of $X$ onto the distribution $\mathfrak h =\{X \in TM|\, X \bot \xi\}$. Inserting this into the previous equation yields
\begin{equation*} \label{4.15}
\si A \Xq + 2h\p A \Xq = 0.
\end{equation*}

\noindent In addition, applying $\p$ to this equation, it follows
\begin{equation*} \label{4.16}
\si \p A \Xq - 2h A \Xq =0.
\end{equation*}
Thus we obtain
\begin{equation*}
\begin{pmatrix}
      \si & 2h \\
      -2h & \si \\
\end{pmatrix}
\begin{pmatrix}
    A \Xq \\
    \p A \Xq \\
\end{pmatrix}
= \begin{pmatrix}
    0 \\
    0 \\
  \end{pmatrix}.
\end{equation*}
The determinant of the square matrix of order $2$, that is, $\si^{2}+4h^{2} \geq \si^{2}\neq 0$, so we get $A\Xq=0$ for any $\Xq \in \mathfrak h$. Substituting~$\Xq$ as $\Kt$ and $\Ks$, it implies $A\Kt=0$ and $A\Ks=0$, respectively. Hence, we can assert that the distribution $\QP$ is invariant under the shape operator, that is, $M$ is a $\QP$-invariant real hypersurface. Thus by virtue of Theorem~$\rm A$, we conclude that $M$ with our assumptions must be a model space of Type~$(A)$.

\vskip 5pt

Summing up these discussions, we conclude that if a Hopf hypersurface $M$ in complex two-plane Grassmannians $\GBt$, $m \geq 3$, satisfying \eqref{C-1} and $\al \neq 2k$, then $M$ is of Type~$(A)$ or $(B)$.

\vskip 3pt

Hereafter, let us check whether $S$ of a model space of Type~$(A)$ (or of Type~$(B)$) satisfies the Reeb parallelism with respect to $\gtw$ by \cite[Proposition~$\rm 3$]{BS1} (or \cite[Proposition~$\rm 2$]{BS1}, respectively).

\vskip 3pt

Let us denote by $M_{A}$ a model space of Type~$(A)$. From now on, using the equations \eqref{3.1}, \eqref{3.3} and \cite[Proposition~$\rm 3$]{BS1}, let us check whether or not $S$ satisfies \eqref{3.7} which is equivalent to our condition \eqref{C-1} for each eigenspace $T_{\alpha}$, $T_{\beta}$, $T_{\lambda}$, and $T_{\mu}$ on $T_{x}M_{A}$, $x \in M_{A}$. In order to do, we find one equation related to $S$ from \eqref{3.7} using the property of $M_{A}$, $\x =\xo$ as follows.
\begin{equation} \label{A1}
\begin{split}
(\gtw_{\x} S)X & = - h(\nabla_{\x}A)X + (\nabla_{\x}A)AX + A(\nabla_{\x}A)X + kh\Ph AX\\
&  \quad \  - k\Ph A^2 X - khA\Ph X + kA^2 \Ph X,
\end{split}
\end{equation}
since $h=\al + 2\beta + 2(m-2)\lambda$ is a constant.

\vskip 3pt

\noindent {\bf Case A-1\,:}\quad $X=\xi(=\xi_1) \in T_\alpha$.

\vskip 3pt

Since $(\N_{\x}A)\x = 0$, we see that $(\gtw_{\x} S)\x=0$ from the equation~\eqref{A1}. It means that the Ricci
tensor $S$ becomes GTW Reeb parallel on $T_{\al}$.

\vskip 3pt

\noindent {\bf Case A-2\,:}\quad $X \in T_{\beta} = \text{Span}\{ \xt, \xi_{3}\}$.

\vskip 3pt

For $\xm \in T_{\beta}$, $\mu=2,3$ we have
\begin{equation*}
\begin{split}
(\N_{\x}A)\xm & = \beta (\N_{\x}\xm) - A (\N_{\x}\xm)\\
              & = \beta \qmt(\x) \xmo - \beta \qmo(\x) \xmt + \al \beta \pmx \\
                & \quad - \qmt(\xi) A\xmo + \qmo(\x)A\xmt - \al A \pmx,
\end{split}
\end{equation*}
which follows that $(\N_{\x}A)\xt=0$ and $(\N_{\x}A)\xh=0$. Therefore, from the equation~\eqref{A1} we obtain, respectively,
\begin{equation*}
\begin{split}
(\gtw_\xi S) \xtw & = kh\p A\xtw - k\p A^{2}\xtw - khA\p\xtw + kA^{2}\p\xtw \\
                  & = (-kh\beta + k\beta^{2} + kh\beta - k\beta^{2})\xth =0,
\end{split}
\end{equation*}
and
$(\gtw_{\xi} S) \xth =0$ by similar methods. So, we assert that the Ricci tensor $S$ of~$M_{A}$ is Reeb parallel on $T_{\beta}$.

\vskip 3pt
By the structure of a tangent vector space $T_{x}M_{A}$ at $x\in M_{A}$, we see that the distribution $\Q$ is composed of two eigenspaces $T_{\lambda}$ and $T_{\mu}$. On this distribution $\Q = T_{\lambda} \oplus T_{\mu}$ we obtain
\begin{equation}\label{eq: 4.8}
(\N_{\x}A)X=\al \p A X - A \p AX + \p X + \po X
\end{equation}
by virtue of the Codazzi equation~\cite[Section~$2$]{LCW}. Using this equation we consider the following two cases.

\vskip 3pt

\noindent {\bf Case A-3\,:}\quad $X \in T_\lambda=\{\,X\,| \ X \in \Q,\ JY=J_1 Y\,\}$.

\vskip 3pt

We naturally see that if $X \in T_{\lambda}$, then $\phi X
= \phi_{1} X$. Moreover, the vector $\phi X$ also belong to the
eigenspace $T_{\lambda}$ for any $X\in T_{\lambda}$, that is, $\phi
T_{\lambda} \subset T_{\lambda}$. From these and \eqref{eq: 4.8}, we obtain
\begin{equation*}
(\N_{\x}A)X= (\al \lambda - \lambda^{2} + 2) \p X, \ \text{for}\ X \in T_{\lambda}.
\end{equation*}
From \eqref{A1} and together with these facts, we obtain
\begin{equation*}
(\gtw_\xi S) X = (\al \lambda - \lambda^{2} +2) (2\al -h) \p X,
\end{equation*}
which implies that $S$ must be Reeb parallel for $\gtw$ on $T_{\lambda}$, since $(\al \lambda - \lambda^{2} +2)=0$.

\vskip 3pt

\noindent {\bf Case A-4\,:}\quad $X \in T_\mu=\{\,X\,| \ X \in {\Q},\ JY=-J_1 Y\,\}.$

\vskip 3pt

If $X \in T_{\mu}$, then $\phi X = -\phi_{1}X$, $\phi T_{\mu} \subset T_{\mu}$ and $\mu =0$. So, from~\eqref{eq: 4.8}, we obtain $(\N_{\x}A)X=0$, moreover $(\gtw_{\xi}S)X = 0$ for any $X \in T_{\mu}$.

\vskip 5pt

\noindent Summing up all cases mentioned above, we can assert that $S$ of real hypersurfaces $M_{A}$ of Type~$(A)$ in $\GBt$ is GTW Reeb parallel.

\vskip 5pt

Now let us consider our problem for a model space of Type~$(B)$, which will be denoted by $M_{B}$. In order to do this, let us calculate the fundamental equation related to the covariant derivative of $S$ of $M_{B}$ along the direction of $\x$ in GTW connection. On $T_{x}M_{B}$, $x \in M_{B}$, since $\x \in \Q$ and $h=\text{Tr}(A)=\al+(4n-1)\beta$ is a constant, the equation \eqref{3.7} is reduced to
\begin{equation*}\label{eq: 4.9}
\begin{split}
(\gtw_{\x}S)X & =  4(k-\al) \sum_{\nu =1}^3 \Big \{\eta_{\nu}(\p X)\x_{\nu} -  \EN(X)\pn\x \Big\} \\
                & \ \ - h(\nabla_{\x}A)X + (\nabla_{\x}A)AX + A(\nabla_{\x}A)X \\
                & \ \ + kh\Ph AX - k\Ph A^2 X - khA\Ph X + kA^2 \Ph X.
\end{split}
\end{equation*}
Moreover, by the equation of Codazzi and \cite[Proposition~$\rm 2$]{BS1} we obtain that for any $X \in T_{x}M_{B}$
\begin{equation}\label{eq: 4.10}
\begin{split}
(\N_{\x}A)X & =\al \p AX - A \p AX + \p X - \sum_{\nu=1}^{3}\big\{\e_{\nu}(X)\p_{\nu}\x + 3g(\p_{\nu}\x, X)\x_{\nu}\big\}\\
            & = \left\{ \begin{array}{ll}
                0                       & \mbox{if}\ \  X \in T_{\al}\\
                \al \beta \p \x_{\ell} & \mbox{if}\ \  X \in T_{\beta}=\text{Span}\{\x_{\ell}|\, \ell=1,2,3\}\\
                -4 \x_{\ell}           & \mbox{if}\ \  X \in T_{\gamma}=\text{Span}\{\p \x_{\ell}|\, \ell=1,2,3\}\\
                (\al \lambda + 2) \p X  & \mbox{if}\ \  X \in T_{\lambda}\\
                (\al \mu + 2) \p X      & \mbox{if}\ \  X \in T_{\mu}.\\
\end{array}\right.
\end{split}
\end{equation}
From these two equations, it follows that
\begin{equation}\label{B1}
(\gtw_{\x}S)X  = \left\{ \begin{array}{ll}
                0                                       & \mbox{if}\ \  X=\xi \in T_{\al}\\
                (\al-k)(4-h \beta + \beta^{2})\p \x_{\ell} & \mbox{if}\ \  X=\x_{\ell} \in T_{\beta} \\
                (4(\al -k) + (h-\beta)(4+k\beta))\x_{\ell} & \mbox{if}\ \ X=\p \x_{\ell} \in T_{\gamma}\\
                (h-\beta)(k \lambda - k \mu - \al \lambda -2)\p X  & \mbox{if}\ \  X \in T_{\lambda}\\
                (h-\beta)(k \mu - k \lambda - \al \mu -2)\p X      & \mbox{if}\ \  X \in T_{\mu}.\\
\end{array}\right.
\end{equation}
So, we see that $M_{B}$ has Reeb parallel GTW-Ricci tensor, when $\al$ and $h$ satisfies the conditions $\al = k$ and $h-\beta=0$, which means $r=\frac{1}{2}\cot^{-1}(\frac{-k}{4(2n-1)})$. Moreover, this radius $r$ satisfies our condition $\al \neq 2k$.

\vskip 5pt

Hence summing up these considerations, we give a complete proof of our Theorem~$1$ in the introduction.~\hspace{8.3cm}$\Box$

\vspace {0.15in}
\section{Proofs of Corollary~$1$ and Theorem~$2$}\label{section 5}
\setcounter{equation}{0}
\renewcommand{\theequation}{5.\arabic{equation}}
\vspace{0.13in}

In section~\ref{section 4} we obtained the classification of Hopf hypersurfaces $M$ with GTW-Reeb parallel Ricci tensor and $\al \neq 2k$. Thus in present section we will consider the case $\al = 2k$ related to the GTW-Reeb parallelism of Ricci tensor of a Hopf hypersurface $M$ in $\GBt$, $m \geq 3$.

\vskip 3pt

Now let us prove Corollary~$1$ in the introduction.

\vskip 3pt

Our condition $\al = 2k$ means that $\al$ is constant. From this we assert that $\x$ belongs to either $\Q$ or $\QP$. For $\x \in \Q$, it is a well-known fact that a Hopf hypersurface in $\GBt$, $m\geq 3$ must be a model space $M_{B}$ of Type~$(B)$ (see \cite{LS}). On the other hand, from~\eqref{B1} and $\al=2k$, the GTW covariant derivative of Ricci tensor $S$ of $M_{B}$ along the direction of $\x$ is given
\begin{equation}\label{eq: 5.1}
(\gtw_{\x}S)X  = \left\{ \begin{array}{ll}
                0                                       & \mbox{if}\ \  X=\xi \in T_{\al}\\
                k(4-h \beta + \beta^{2})\p \x_{\ell} & \mbox{if}\ \  X=\x_{\ell} \in T_{\beta} \\
                (4k + (h-\beta)(4+k\beta))\x_{\ell} & \mbox{if}\ \ X=\p \x_{\ell} \in T_{\gamma}\\
                -(h-\beta)(k \beta + 2)\p X  & \mbox{if}\ \  X \in T_{\lambda}\\
                -(h-\beta)( k \beta + 2)\p X      & \mbox{if}\ \  X \in T_{\mu}.\\
\end{array}\right.
\end{equation}
Actually, since $\al = 2k$, we naturally have $k \beta + 2=0$. It follows that $S$ is GTW Reeb parallel on $T_{\lambda}$ and $T_{\mu}$. In order to be the GTW-Reeb parallel Ricci tensor on the other eigenspaces $T_{\beta}$ and $T_{\gamma}$, we should have the following two equations,
$$
(4-h \beta + \beta^{2})=0
$$
and
$$
4k + (h-\beta)(4+k\beta)=0.
$$
 Combining these two equations, we have $2k + h - \beta=0$. Since $h=\al + 3 \beta + (4n-4) (\lambda + \mu) = \al +(4n-1)\beta$ and $\al = 2k$, it follows that $\al = -(2n-1)\beta$. By virtue of \cite[Proposition~$\rm 2$]{BS1}, $\al = -2 \tan(2r)$ and $\beta = 2 \cot (2r)$ where $r \in ( 0, \pi/4)$, we obtain $\tan (2r)=\sqrt{2n-1}$. From such assertions, we conclude that a model space of Type~$(B)$ has GTW-Reeb parallel Ricci tensor for special radius $r$ such that $r=\frac{1}{2}\tan^{-1}(\sqrt{2n-1})$, which gives us a complete proof of Corollary~$1$. \hspace{7.7cm}$\Box$

\vskip 7pt

On the other hand, for the case $\x \in \QP$, the equation \eqref{4.6} becomes
\begin{equation*}
(\x h) AX = 0
\end{equation*}
under the assumption of $\al = 2k$. For the case $\x h\neq 0$, it follows that $AX=0$. If $X=\x$, then $\al=0$, which gives a contradiction. From this, we assert the following for the case $\x \in \QP$:
\begin{rem}
Let $M$ is a Hopf hypersurface, that is, $A\x=\al \x$ where $\al = 2k$, in $\GBt$, $m \geq 3$, with GTW-Reeb parallel Ricci tensor, $\gtw_{\x}S=0$. If $\x \in \QP$, then we only get the result that the trace $h$ of the shape operator $A$ is constant along the direction of $\x$, that is, $\x h=0$.
\end{rem}

From such a point of view, we now only focus our attention to the Ricci Reeb parallelism in GTW connection on the distribution~$\mathfrak h=\{X \in TM\,|\, X \bot \xi \}$, as given by the proof of Theorem~$2$.

\vskip 3pt

As mentioned above in the proof of Corollary~$1$, we see that $\x \in \Q$ or $\x \in \QP$, because $M$ is a Hopf hypersurface in $\GBt$ with $\al = 2k$. Moreover, if $\x \in \Q$, then~$M$ must be a model space of Type~$(B)$.

\vskip 5pt

Now, let us consider the case $\x \in \QP$. Then by Suh~\cite{S01} we have the following key lemma in the proof of Theorem~$2$.
\begin{lemma}\label{lemma 5.1}
Let $M$ be a Hopf hypersurface, that is, $A\x=\al \x$ where $\al=2k$, in $\GBt$, $m \geq 3$. If $M$ satisfies the following properties\,:
\begin{enumerate}[\rm (i)]
\item{the Reeb vector field $\x$ belongs to the distribution $\QP$,}
\item{the Ricci tensor $S$ is Reeb parallel with respect to both the Levi-Civita and GTW connections on $\mathfrak h$, that is, $(\gtw_{\x}S)X=0$ and $(\N_{\x}S)X=0$ for any tangent vector field $W \in \mathfrak h$,}
\end{enumerate}
then $M$ must be a model space of Type~$(A)$ or Type~$(B)$ in $\GBt$.
\end{lemma}
\begin{proof}
As investigated above, from the assumption of $\al = 2k$ and the equation~\eqref{4.6} we have
$$
(\x h) AW = 0
$$
for any tangent vector field $W \in \mathfrak h$.

From this, we see that the distribution~$\mathfrak h$ is totally geodesic, that is, $AW=0$ for any $W \in \mathfrak h$, if $(\x h) \neq 0$. So, we can assert that $M$ is a $\QP$-invariant hypersurface in $\GBt$, that is, $g(A\Q, \QP)=0$.

\vskip 3pt

Next, we consider the case $(\x h)=0$. From \eqref{3.1} we get $S\x = (4m+ h\al - \al^{2})\x$. Differentiating this formula along the direction of $\x$ and using our assumptions, $A\x = \al \x$, $(\x h) = (\x \al) = 0$, it follows that $(\N_{\x}S)\x = 0$. It implies that the Ricci tensor $S$ becomes Reeb parallel. Then by virtue of the result given by Suh~\cite{S01} we give a complete proof of our Lemma.
\end{proof}

As a consequence, we assert that if $M$ is a Hopf hypersurface, $\al = 2k$, in $\GBt$ satisfying two Ricci Reeb parallelism defined by $(\N_{\x}S)W=0$ and $(\gtw_{\x}S)W=0$ for any $W \in \mathfrak h$, then it must be either a real hypersurface of Type~$(A)$ or Type~$(B)$.

\vskip 3pt

From now on, let us consider the converse problem. In other words, we now check whether the Ricci tensor $S$ of model spaces $M_{A}$ or $M_{B}$ in $\GBt$ satisfies the conditions in Theorem~$2$ or not.

\vskip 3pt

By \cite[Proposition~$\rm 3$]{BS1} and the checking for a model space $M_{A}$ given in the introduction and section~\ref{section 4}, respectively, we see that $M_{A}$ is a Hopf hypersurface in $\GBt$ with the GTW-Reeb parallel Ricci tensor on $\mathfrak h \subset TM_{A}$.

\vskip 3pt

Now let us show that the Ricci tensor $S$ of $M_{A}$ is Reeb parallel in $\N$ on $\mathfrak h$, that is, $(\N_{\x}S)W=0$ for $W \in \mathfrak h \subset TM_{A}$. By virtue of \cite[Proposition~$\rm 3$]{BS1}, the equation~\eqref{3.4} can be written as
\begin{equation}\label{5.1}
\begin{split}
(\N_{\x}S)Y & = h(\N_{\x}A)Y - (\N_{\x}A)AY - A(\N_{\x}A)Y  \\
            & = h(\N_{\x}A)Y - \tilde{\kappa}(\N_{\x}A)Y - A(\N_{\x}A)Y,
\end{split}
\end{equation}
where $AY=\tilde{\kappa} Y$ for any $W \in \mathfrak h\subset TM_{A}$. Moreover, from the equation of Codazzi, we obtain
\begin{equation}\label{5.2}
\begin{split}
(\N_{\x}A)Y & = (\N_{Y}A)\x + \p Y + \po Y - 2\e_{2}(Y)\xh + 2\e_{3}(Y)\xt \\
            & = \al \p AY - A \p AY + \p Y + \po Y - 2\e_{2}(Y)\xh + 2\e_{3}(Y)\xt,
\end{split}
\end{equation}
since $A\x=\al \x$ and $h=\text{Tr}A = \al + 2 \beta + (2m-2)\lambda$ where the eigenvalues $\al$, $\beta$, $\lambda$ and $\mu$ of $M_{A}$ are constant. Since $\mathfrak h=T_{\beta}\oplus T_{\lambda} \oplus T_{\mu}$, let us check whether or not the Ricci tensor~$S$ of~$M_{A}$ satisfies the property of the Reeb-parallelism for each eigenspace.

\vskip 3pt

\noindent {\bf Case A-1\,:}\quad $Y \in T_\beta = \text{Span}\{ \xt, \xh\}$

\vskip 3pt

From \eqref{5.2}, we obtain $(\N_{\x}A)\xt  = (\beta^{2} - \al \beta  -2)\xh $, which implies $(\N_{\x}A)\xt =0$ since $\beta^{2} - \al \beta  -2 =0$. So, we see that $(\N_{\x}S)\xt=0$ by \eqref{5.1}. Similarly, if we put $Y=\xh$ in \eqref{5.2}, then $(\N_{\x}A)\xh = -(\beta^{2} - \al \beta  -2)\xt=0$, because $\al \beta=2 \cot^{2}(\sqrt{2}r)-2$. From this and \eqref{5.1}, we see that $(\N_{\x}S)\xh=0$.

\vskip 3pt

\noindent {\bf Case A-2\,:}\quad $Y \in T_\lambda = \{ Y \bot \, \xo, \xt, \xh \, |\, \p Y =\po Y\}$

\vskip 3pt

If $Y \in T_{\lambda}$, then $\p Y \in T_{\lambda}$. From this, the equation~\eqref{5.2} becomes $(\N_{\x}A)Y = (\al \lambda - \lambda^{2}+ 2) \p Y$. It follows $(\N_{\x}A)Y = 0$, since $\al \lambda = 2 \tan^{2}(\sqrt{2}r)-2$. Hence we see that the Ricci tensor $S$ of $M_{A}$ becomes Reeb parallel on $T_{\lambda}$, that is, $(\N_{\x}S)Y=0$ for any $Y \in T_{\lambda}$.

\vskip 3pt

\noindent {\bf Case A-3\,:}\quad $Y \in T_\mu = \{ Y \bot \, \xo, \xt, \xh \, |\, \p Y =-\po Y\}$

\vskip 3pt

Since $Y \in T_{\mu}$, then $\p Y = - \po Y$ and $\mu=0$. From these, the equation~\eqref{5.2} becomes $(\N_{\x}A)Y = 0$ for $Y \in T_{\mu}$. Hence it implies that the Ricci tensor of $M_{A}$ is Reeb parallel on $T_{\mu}$, that is, $(\N_{\x}S)Y=0$ for any $Y \in T_{\mu}$.

\vskip 3pt

\noindent Summing up three cases above, $M_{A}$ have Reeb parallel Ricci tensor in the Levi-Civita connection $\N$ on the distribution $\mathfrak h$.

\vskip 5pt

On the other hand, let us check whether $M_{B}$ satisfies our conditions, $\N_{\x}S=0$ and $\gtw_{\x}S=0$ on $\mathfrak h \subset TM_{B}$. Suppose that the Ricci tensor~$S$ of $M_{B}$ is Reeb parallel, $(\N_{\x}S)X=0$ for $X \in \mathfrak h$. From \eqref{3.4} and \eqref{eq: 4.10} we obtain
\begin{equation*}
(\N_{\x}S)X  = \left\{ \begin{array}{ll}
                (-4 \al + h\al \beta - \al \beta^{2})\p \x_{\ell} & \mbox{if}\ \  X=\x_{\ell} \in T_{\beta} \\
                -4(\al + h - \beta) \x_{\ell}                     & \mbox{if}\ \ X=\p \x_{\ell} \in T_{\gamma}\\
                (h-\beta)(\al \lambda + 2)\p X                    & \mbox{if}\ \  X \in T_{\lambda}\\
                (h-\beta)(\al \mu + 2)\p X                        & \mbox{if}\ \  X \in T_{\mu}.\\
\end{array}\right.
\end{equation*}
Since the Ricci tensor~$S$ is Reeb parallel on the eigenspace $T_{\lambda}$, we have $(h-\beta)(\al \lambda + 2)=0$. It implies that \begin{equation}\label{eq: 5.4}
(h-\beta)=0,
\end{equation}
because $(\al \lambda +2) \neq 0$. On the other hand, for $T_{\gamma}$ we get $(\al + h - \beta)=0$, which means $\al =0$ from \eqref{eq: 5.4}. It makes a contraction. Thus we assert that there does not exist $M_{B}$ satisfying the conditions in Theorem~$2$.

%
%
\vskip 5pt

With such assertions we give a complete proof of Theorem~$2$ in the introduction.~\hspace{11.5cm} $\Box$

\vspace {0.15in}
\section{Proofs of Theorem~$3$ and Corollary~$2$}\label{section 6}
\setcounter{equation}{0}
\renewcommand{\theequation}{6.\arabic{equation}}
\vspace{0.13in}

First we want to give a proof of Theorem~$3$. Among the conditions in Theorem~$2$, we focus our attentions to the assumptions related to the Reeb parallelism of Ricci tensor $S$. Actually, we consider that on $\mathfrak h$ two covariant derivatives of $S$ in Levi-Civita and GTW connections are equal to zero, that is, $(\N_{\x}S)W = 0 = (\gtw_{\x}S)W$ for any tangent vector field $W \in \mathfrak h =\{X\in TM|\,X \bot \x \}$. So, in this section, we will consider the following condition related to the Reeb parallelism of Ricci tensor $S$.
\begin{equation}\label{C-2}
(\N_{\xi}S)X=(\gtw_{\xi} S)X
\tag{C-2}
\end{equation}

for any tangent vector field $X$ on $M$. By virtue of the equation~\eqref{3.5}, the condition~\eqref{C-2} is equivalent to the $S \p = \p S$. On the other hand, Suh proved in \cite{S02} that a Hopf hypersurface $M$ in $\GBt$, $m \geq 3$, with commuting Ricci tensor is locally congruent a tube of radius $r$ over a totally geodesic $\GBo$ in $\GBt$. Then we conclude that a Hopf hypersurface $M$ in $\GBt$, $m\geq 3$, satisfying the condition~\eqref{C-2} if and only if $M$ is of Type~$(A)$, which gives us a complete proof of Theorem~$3$.

\vskip 5pt

By Theorem~$3$, if a real hypersurface $M$ in $\GBt$ satisfies $\N S=\gtw S$, then naturally \eqref{C-2} holds on $M$. So $M$ is of Type~$(A)$. Now let us check whether a model space $M_{A}$ of Type~$(A)$ satisfies our condition
\begin{equation}\label{C-3}
(\gtw_{X} S)Y=(\N_{X} S)Y
\tag{C-3}
\end{equation}
for any tangent vector fields $X$, $Y \in T_{x}M_{A}$, $x \in M_{A}$. In order to do this, we assume that the Ricci tensor $S$ of $M_{A}$ satisfies~\eqref{C-3}. That is, we have
\begin{equation}\label{eq: 6.1}
\begin{split}
0& =(\gtw_{X} S)Y-(\N_{X} S)Y \\
 & = g(\phi AX, SY)\x - \eta(SY) \phi AX - k\eta(X)\phi SY \\
 & \quad \quad - g(\phi AX,Y)S \x + \eta(Y)S\phi AX + k\eta(X)S\phi Y
\end{split}
\end{equation}
for any $X, Y \in T_{x}M_{A}$.

Since $T_{x}M_{A}=T_{\al} \oplus T_{\beta} \oplus T_{\lambda} \oplus T_{\gamma}$, the equation~\eqref{eq: 6.1} holds for $X\in T_{\beta}$ and $Y\in T_{\al}$. For the sake of convenience we put $X=\xt \in T_{\beta}$ and $Y=\xi \in T_{\al}$. Since $S\x=\delta \x$ and $S \xh =\sigma \xh$ where $\delta=(4m+h\al-\al^{2})$ and $\sigma=(4m+6+h\beta-\beta^{2})$, the equation~\eqref{eq: 6.1} reduces to $\beta(\delta - \sigma) \xh=0$. By \cite[Proposition~$\rm 3$]{BS1}, since the principal curvature $\beta=\sqrt{2} \cot(\sqrt{2}r)$ for $r \in (0, \pi/\sqrt{8}$) is non-zero, it follows $(\delta - \sigma)=0$. In other words, by \cite[Proposition~$\rm 3$]{BS1} we obtain
\begin{equation*}
\begin{split}
-(\delta - \sigma)& = 6- \al\beta + \beta^{2} + (2m-2)\beta \lambda - (2m-2)\al \lambda \\
                  & = 8-4(m-1)\tan^{2}(\sqrt{2}r),\\
\end{split}
\end{equation*}
which gives us
\begin{equation}\label{eq: 6.2}
\tan^{2}(\sqrt{2}r)=\frac{2}{m-1}.
\end{equation}
In addition, since \eqref{eq: 6.1} holds for $X \in T_{\lambda}$ and $Y=\x$, we obtain
\begin{equation*}
0 = (\gtw_{X}S)\x - (\N_{X}S)\x = \lambda (\tau-\delta)\p X,
\end{equation*}
where in the second equality we have used $\p X \in T_{\lambda}$ and $SX = (4m+6+h\lambda-\lambda^{2})X=\tau X$ for any $X \in T_{\lambda}$. Because $\lambda=-\sqrt{2}\tan(\sqrt{2}r)$ where $r \in (0, \pi/\sqrt{8})$ is non-zero, we have also
$$
\tau-\delta=0.
$$
By straightforward calculation it is
\begin{equation*}
\begin{split}
\tau-\delta & = 6+h\lambda-\lambda^{2}-h\al+\al^{2} \\
            & = 4m - 4 \cot^{2}(\sqrt{2}r)=0.
\end{split}
\end{equation*}
From \eqref{eq: 6.2}, it becomes $2m+2=0$, which gives us a contradiction. Accordingly, it completes our Corollary~$2$ given in the introduction. \hspace{4.5cm}$\Box$

\vskip 15pt
%


\end{document}